\documentclass[11pt, reqno]{amsart}
\usepackage{amsmath}
\usepackage{amssymb}
\usepackage{esint, enumitem}
\usepackage[hidelinks]{hyperref}
\newtheorem{theorem}{Theorem}[section]
\newtheorem{lemma}[theorem]{Lemma}
\newtheorem{proposition}[theorem]{Proposition}

\newtheorem{problem}[theorem]{Problem}

\theoremstyle{definition}
\newtheorem{definition}[theorem]{Definition}

\theoremstyle{remark}
\newtheorem{remark}[theorem]{Remark}
\newtheorem{claim}{Claim}

\numberwithin{equation}{section}

\newcommand{\bb}[1]{\mathbb{#1}}
\newcommand{\innerproduct}[2]{\langle\, #1,#2 \,\rangle}

\allowdisplaybreaks[4]

\title[Hessian Estimates for the Sigma-2 Equation]{Hessian Estimates for the Sigma-2 Equation with Variable Right-Hand Side Terms in Dimension 4}

\date{\today}

\author{Zhenyu Fan}
\address{School of Mathematical Sciences, Peking University, Beijing, 100871, P. R. China}
\email{fanzhenyu@stu.pku.edu.cn}

\begin{document}


\begin{abstract}
    We derive a priori interior Hessian estimates and regularity  for the sigma-2 Hessian equation $\sigma_{2}(D^2u)=f(x,u,Du)$ with positive $C^{1,1}$ right hand side in dimension 4. In higher dimensions, the same result holds under an additional dynamic semi-convexity condition on solutions. This  generalizes Qiu's and Shankar-Yuan's results \cite{Qiu-sigma2Hessian-24, Shankar-Yuan-25}.
\end{abstract}

\maketitle

\section{Introduction}
The $\sigma_{k}\,$-Hessian equation
\[ \sigma_{k}(D^2u)=\sum_{1\leq i_{1}<\cdots<i_{k}\leq n}\lambda_{i_{1}}\cdots\lambda_{i_{k}}=1, \]
where $\lambda_{i}'s$ are eigenvalues of the Hessian $D^2u$, is a class of widely studied fully nonlinear elliptic equations due to its relatively simple structure and broad applications in geometry and other fields. For $k=1$, $\sigma_{1}(D^2u)=\Delta u=1$ reduces to the Laplace equation. For $k=n$, $\sigma_{n}(D^2u)=\det D^2u=1$ corresponds to the Monge-Amp\`ere equation which originates from optimal transport theory and prescribed Gaussian curvature problem.

 A priori estimates and regularity theory constitute a fundamental subject in the study of partial differential equations. It is well known that viscosity solutions to the Laplace equation $\Delta u=1$ are smooth (even analytic) and admit appropriate estimates.  For the Monge-Amp\`ere equation $\det D^2u=1$,  in two dimensions, Hessian estimates was first established by Heniz \cite{Heinz-59}. In higher dimensions $n\ge3$, Pogorelov's $C^{1,1-\frac{2}{n}}$ counterexamples \cite{Pogorelov-cg-71,Pogorelov-book} preclude establishing a priori interior Hessian estimates and $C^2$ regularity.  Later, Urbas \cite{Urbas} extended Pogorelov's counterexamples to $\sigma_{k}\,$-Hessian equations with $k\ge 3$. This leaves a longstanding open problem:

 \begin{problem}
        For $n\ge3$, whether a priori interior Hessian estimates and regularity can be established for the $\sigma_{2}\,$-Hessian equation $\sigma_{2}(D^2u)=1$?
    \end{problem}

Now this problem is resolved in dimensions $n=3$ by Warren-Yuan \cite{Warren-Yuan-3d-CPAM}, $n=4$ by Shankar-Yuan \cite{Shankar-Yuan-25} and remains open in higher dimensions $n\ge 5$. Furthermore, in dimension $n=3$, a priori Hessian estimates for $\sigma_{2}(D^2u)=f$ with general right-hand side terms were derived by Qiu \cite{Qiu-sigma2Hessian-24},  Zhou \cite{ZhouXC} and Xu \cite{XuYF}.

In this article, we generalize the results of Qiu \cite{Qiu-sigma2Hessian-24} and Shankar-Yuan \cite{Shankar-Yuan-25}. We establish a priori interior Hessian estimates and regularity for the $\sigma_{2}\,$-Hessian equation with general right-hand side terms 
\begin{equation}\label{sigma2 eq}
    F(D^2u)=\sigma_{2}(\lambda(D^2u))=\sum_{i<j }\lambda_{i}\lambda_{j}=f(x,u,Du)
\end{equation}
 in dimension $n=4$, where $f=f(x,z,p) $ is a positive $C^{1,1}$ function on $B_{1}\times \bb{R}\times \bb{R}^n$. Our main result is as follows:

\begin{theorem}\label{THM: main thm}
    Let $u$ be a smooth solution to \eqref{sigma2 eq} in the positive branch $\Delta u>0$ on $B_{1}\subset\bb{R}^4$. Then we have the implicit Hessian estimate
    \[ |D^2u(0)|\leq C, \]
    where $C$ is a universal constant depending only on $\|f\|_{C^{1,1}}$, $\|\frac{1}{f}\|_{L^{\infty}}$ and $\|u\|_{C^{1}(B_{1})}$.
\end{theorem}

In higher dimensions, we can also prove a Hessian estimate for solutions with an additional dynamic semi-convexity condition.

\begin{theorem}\label{THM: est for dynamic semi-convex solus}
    For $n\ge 5$, let $u$ be a smooth solution to \eqref{sigma2 eq} in the positive branch $\Delta u>0$ on $B_{1}\subset\bb{R}^n$.
    Assume that $u$ satisfies the following dynamic semi-convex condition: 
    \begin{align}\label{eq: dynamic semi-convex condition}
        \lambda_{\text{min}}(D^2u)\ge -c(n)\Delta u\quad with\quad c(n)=\dfrac{\sqrt{3n^2+1}-n+1}{2n},
    \end{align}
    where $\lambda_{min}(D^2u)$ denotes the minimum eigenvalue of $D^2u$.
    
    Then we have the implicit Hessian estimate
     \[ |D^2u(0)|\leq C, \]
    where $C$ is a universal constant depending only on $n$, $\|f\|_{C^{1,1}}$, $\|\frac{1}{f}\|_{L^{\infty}}$ and $\|u\|_{C^{1}(B_{1})}$.
\end{theorem}

One application of the above a priori interior Hessian estimates is the interior regularity for viscosity solutions to $\sigma_{2}=f$ in dimension 4. 

\begin{theorem}
    Let $u$ be a continuous viscosity solution to $\sigma_{2}(D^2u)=f$ on $B_1\subset\bb{R}^4$ with $\Delta u>0$ in the viscosity sense. Suppose that $f\in C^{1,1}(B_1)$ and $f>0$. Then $u\in C^{3,\alpha}_{\mathrm{loc}}(B_{1})$ for any $\alpha\in(0,1)$.
\end{theorem}

Let us briefly review the ideas for establishing Hessian estimates for the $\sigma_{2}\,$-Hessian equations so far. In two dimensional case, Heinz \cite{Heinz-59} derived the Hessian estimate for $\sigma_{2}=\det D^2u=1$ via isothermal coordinates. In recent years, Chen-Han-Ou \cite{Chen-Han-Ou} provided a maximum principle proof, and Liu \cite{LiuJK} gave an alternative proof via the partial Legendre transform.

In three dimensions, the equation $\sigma_{2}=1$ corresponds to the special Lagrangian equation with critical phase $\arctan\lambda_{1}+\arctan\lambda_{2}+\arctan\lambda_{3}=\pm\pi/2$. Now the gradient graph $(x,Du(x))$ is a minimal surface (in fact volume minimizing) in $\bb{R}^3\times\bb{R}^3$ \cite{Harvey-Lawson}. By the help of this minimal surface structure, Warren-Yuan \cite{Warren-Yuan-3d-CPAM} derived a Hessian estimate via integral methods. Regarding variable right-hand sides $\sigma_{2}=f$, Qiu \cite{Qiu-sigma2Hessian-24}  observed a similar structure that the gradient graph $(x,Du(x))$ has bounded mean curvature in $(\bb{R}^3\times \bb{R}^3, f^2\mathrm{d}x^2+\mathrm{d}y^2)$ and used this structure to derive a Hessian estimate for $C^{1,1}$ positive $f$. For Lipschitz $f$, Zhou \cite{ZhouXC} further investigated the generalized special Lagrangian equation $\sum \arctan \lambda_{i}/f=\Theta$, where the three dimensional $\sigma_{2}=f$ is just a special case. Moreover, Xu \cite{XuYF} derived interior $C^{2,\alpha}$ estimates only under a small enough H\"older seminorm assumption on $f$.

In four dimensions, Shankar-Yuan's recent breakthrough \cite{Shankar-Yuan-25} introduces a new doubling method for deriving Hessian estimates for $\sigma_{2}=1$. They combined the ideas of Qiu \cite{Qiu-sigma2Hessian-24}, Chaudhuri-Trudinger \cite{Chaudhuri-Trudinger-05} and Savin \cite{Savin-07}. Qiu proved that when the Jacobi inequality
\begin{equation}\label{intro: Jacobi}
    F_{ij}\partial_{ij}\log \Delta u\ge \varepsilon F_{ij}\partial_{i}(\log\Delta u)\partial_{j}(\log\Delta u)
\end{equation}
is valid, a doubling inequality 
\[ \sup_{B_{1}}\Delta u\leq C\left(n, \|u\|_{C^1} \right) \sup_{B_{1/2}}\Delta u \]
can be derived using the maximum principle and Guan-Qiu's test function \cite{Guan-Qiu-Duke-19}. Through rescaling and iteration, this implies that if we can control the Hessian on a small ball, then it can also be controlled at large scales. For three dimensions, the Jacobi inequality \eqref{intro: Jacobi} holds for $\varepsilon=1/3$. In four dimensional case, we can only establish an ``almost Jacobi'' inequality which means that $\varepsilon$ may degenerate to $0$ somewhere. However, this still leads a doubling inequality via a slight modification of Qiu's argument. To obtain Hessian control at small scales, they also proved a partial regularity result. By modifying the proof of the Alexandrov theorem \cite{Evans-Gariepy,Chaudhuri-Trudinger-05}, they proved the almost everywhere twice differentiability of viscosity solutions. Combining this with Savin's small perturbation theorem ($\varepsilon$-regularity theorem) \cite{Savin-07}, the partial regularity follows. Finally, a compactness argument gives us a desired implicit Hessian estimate.

In higher dimensions $n\ge 5$, Hessian estimate under additional assumptions on solutions have been extensively studied. We refer to Guan-Qiu \cite{Guan-Qiu-Duke-19} and Chen-Jian-Zhou \cite{Chen-Jian-Zhou} for the convex case (or under the general assumption $\sigma_{3}(D^2u)\ge-A$ in Guan-Qiu's work), McGonagle-Song-Yuan \cite{McGonagle-Song-Yuan} for the almost convex case, and Shankar-Yuan \cite{Shankar-Yuan-CVPDE-20} for the semi-convex case. Moreover, under the dynamic semi-convex condition \eqref{eq: dynamic semi-convex condition}, the almost Jacobi inequality still holds. Consequently, Hessian estimates can be established using the outlined approach, see \cite{Shankar-Yuan-25}.

For more introductions to recent advances in $\sigma_{2}\,$-Hessian equations, we refer to the survey \cite{Yuan-survey} wrote by Yuan.

In this article, we employ Shankar-Yuan's doubling method to handle the case of $C^{1,1}$ variable right-hand side terms in dimension 4. Here, we face three difficulties. First, when differentiating the equation, the right-hand side terms are no longer zero. This introduces extra negative terms when proving the almost Jacobi inequality. Through careful analysis, we still obtain the almost Jacobi inequality with controllable remainder terms, see Proposition \ref{PROP: Almost Jacobi ineq}. This still leads the doubling inequality. Second, to prove the Alexandrov type theorem, we need Hölder or gradient estimates for viscosity solutions. Here, we use Labutin's potential estimate \cite{Labutin-Duke-02} to achieve this. Third, Savin's small perturbation theorem cannot be directly applied to $F(D^2u,Du,u,x)=\sigma_{2}(D^2u)-f(x,u,Du)=0$ since Savin requires $F(0,0,0,x)\equiv 0$. Thus, a generalized version of Savin's theorem is needed. This generalization is proved by Lian-Zhang \cite{lian-zhang} and my recent work \cite{Fan}.

Once establishing a priori interior Hessian estimates, the interior regularity follows from the standard approximating process using \cite{CNS3-85} and the Evans-Krylov theory.

This article is organized as follows. In Sections \ref{Section: Almost Jacobi Inequality} and \ref{Section: Doubling ineq}, we establish the almost Jacobi inequality and the doubling inequality. In Section \ref{Section: Alexandrov Regularity}, we establish the Alexandrov type regularity theorem for viscosity solutions. In Sections \ref{Section: Genaralized Savin's thm} and \ref{Section: Proof of Main Theorems}, we stated our generalized small perturbation theorem and provide the proof of Theorem \ref{THM: main thm} and \ref{THM: est for dynamic semi-convex solus}.

\bigskip
\noindent
{\bf Acknowledgments.}  The author is supported by National Key R\&D Program of China 2020YFA0712800. He is grateful to Professors Yu Yuan and Ravi Shankar for helpful discussions and encouragements, and to Professors Qing Han and Yuguang Shi for their supports.

\section{Almost Jacobi Inequality}\label{Section: Almost Jacobi Inequality}
A solution $u$ to the sigma-2 equation \eqref{sigma2 eq} is called admissible, or 2-convex, if
\[\lambda(D^2u)\in \Gamma_{2}:=\{\lambda\in\bb{R}^n: \sigma_{1}(\lambda)>0,\sigma_{2}(\lambda)>0\}.\]
We denote $\Delta_{F}=F_{ij}\partial_{ij}$ to be the linearized operator for the $\sigma_{2}$ equation \eqref{sigma2 eq} at $D^2u$, where  
\begin{equation}
    (F_{ij})=\left(\dfrac{\partial F}{\partial u_{ij}}\right)=\Delta u I-D^2u.
\end{equation}
By \cite{CNS3-85, WnagXJ-note-k-Hessian}, \eqref{sigma2 eq} is elliptic, or equivalently  $(F_{ij})$ is positive definite, provided $u$ is admissible.

The gradient square $|\nabla_{F} v|^2$ for any smooth function $v$ with respect to the matrix $(F_{ij})$ is defined by
\[|\nabla_{F}v|^2=\sum_{i,j}F_{ij}\partial_{i}v\partial_{j}v.\]

\subsection{Preliminaries}
We first introduce some algebraic preliminary lemmas related to the $\sigma_{2}$ operator. It can be found in \cite{Li-Trudinger-94,Qiu-sigma2Hessian-24,Shankar-Yuan-25}. For completeness, we present the proof.

The first Lemma gives us a sharp control on the minimum eigenvalue.
\begin{lemma}\label{LEMMA:Sharp estimate for minimal eigenvalue}
    Let $\lambda=(\lambda_{1},\cdots,\lambda_{n})\in\Gamma_{2}=\{\sigma_{1}(\lambda)>0, \sigma_{2}(\lambda)>0\}$ with $\lambda_{1}\ge\lambda_{2}\ge\cdots\ge \lambda_{n}$. Then for $n>2$ the following sharp bound holds:
    \[ \sigma_{1}>\dfrac{n}{n-2}|\lambda_{n}|. \]
\end{lemma}
\begin{proof}
    If $\lambda_{n}\ge 0$, we have $\sigma_{1}\ge n\lambda_{n}$, the conclusion holds.

    If $\lambda_{n}<0$, set $\lambda'=(\lambda_{1},\cdots,\lambda_{n-1})$. Note that $\sigma_{2}(\lambda)=\lambda_{n}\sigma_{1}(\lambda')+\sigma_{2}(\lambda')$, then
    \[ \lambda_{n}=\dfrac{\sigma_{2}(\lambda)-\sigma_{2}(\lambda')}{\sigma_{1}(\lambda')}<0. \]
    In particular, this implies $0<\sigma_{2}(\lambda)<\sigma_{2}(\lambda')$. Now
    \begin{align}\label{eq: est on sigma 1/lambda n}
        \dfrac{\sigma_{1}(\lambda)}{|\lambda_{n}|}=\dfrac{\sigma_{1}(\lambda')}{-\lambda_{n}}-1=\dfrac{[\sigma_{1}(\lambda')]^2}{\sigma_{2}(\lambda')-\sigma_{2}(\lambda)}-1>\dfrac{[\sigma_{1}(\lambda')]^2}{\sigma_{2}(\lambda')}-1.
    \end{align}
    Since $[\sigma_{1}(\lambda')]^2=|\lambda'|^2+2\sigma_{2}(\lambda')$ and 
    \[ 2\sigma_{2}(\lambda')=\sum_{1\leq i<j\leq n-1}2\lambda_{i}\lambda_{j}\leq \sum_{1\leq i<j\leq n-1}(\lambda_{i}^2+\lambda_{j}^2)=(n-2)|\lambda'|^2.  \]
    We conclude that 
    \begin{equation}\label{eq: est on (sigma1)^2}
        [\sigma_{1}(\lambda')]^2\ge \left[\dfrac{2}{n-2}+2\right]\sigma_{2}(\lambda')=\dfrac{2n-2}{n-2}\sigma_{2}(\lambda').
    \end{equation}
    Finally, \eqref{eq: est on sigma 1/lambda n} and \eqref{eq: est on (sigma1)^2} yield
    \[ \dfrac{\sigma_{1}(\lambda)}{|\lambda_{n}|}>\dfrac{2n-2}{n-2}-1=\dfrac{n}{n-2}. \]
\end{proof}

\begin{lemma}\label{LEMMA: est for Fii}
   Let $\lambda=(\lambda_{1},\cdots,\lambda_{n})\in\Gamma_{2}=\{\sigma_{1}(\lambda)>0, \sigma_{2}(\lambda)>0\}$ with $\lambda_{1}\ge\lambda_{2}\ge\cdots\ge \lambda_{n}$, then we have
   \begin{equation}
       \dfrac{\sigma_{2}}{\sigma_{1}}\leq (\sigma_{1}-\lambda_{1})\leq  \dfrac{n-1}{n}\sigma_{1} ,
   \end{equation}
   and for any $i\ge 2$,
   \begin{equation}
       \left(1-\dfrac{1}{\sqrt{2}}\right)\sigma_{1}\leq \sigma_{1}-\lambda_{i}\leq \dfrac{2n-2}{n}\sigma_{1} .
   \end{equation}
    
\end{lemma}
\begin{proof}
    Notice that $\sigma_{1}^2=2\sigma_{2}+|\lambda|^2$, then we have
    \begin{equation*}
        \sigma_{1}-\lambda_{1}=\dfrac{\sigma_{1}^2-\lambda_{1}^2}{\sigma_{1}+\lambda_{1}}\ge \dfrac{2\sigma_{2}}{\sigma_{1}+\lambda_{1}}\ge \dfrac{\sigma_{2}}{\sigma_{1}}.
    \end{equation*}
    For the upper bound, 
    \[ \sigma_{1}-\lambda_{1}\leq \sigma_{1}-\dfrac{1}{n}\sigma_{1}\leq \dfrac{n-1}{n}\sigma_{1}. \]
    
    For $i\ge 2$, if $\lambda_{i}\leq0$, the lower bound for $\sigma_{1}-\lambda_{i}$ is obvious. If $\lambda_{i}>0$, we have
    \begin{equation*}
        \sigma_{1}-\lambda_{i}\ge \sigma_{1}-\sqrt{\dfrac{\lambda_{1}^2+\cdots+\lambda_{i}^2}{i}}\ge \sigma_{1}-\dfrac{|\lambda|}{\sqrt{2}}\ge \left(1-\dfrac{1}{\sqrt{2}}\right)\sigma_{1}.
    \end{equation*}
    The upper bound of $\sigma_{1}-\lambda_{i}$ follows from Lemma \ref{LEMMA:Sharp estimate for minimal eigenvalue},
    \[ \sigma_{1}-\lambda_{i}\leq \sigma_{1}+|\lambda_{n}|\leq \dfrac{2n-2}{n}\sigma_{1}. \]
\end{proof}

\subsection{Almost Jacobi Inequality}
In this subsection, we establish the almost Jacobi inequality for the quantity $b=\log\Delta u$. We begin by proving preliminary estimates on the derivatives of $f$.

\begin{lemma}
        Let $u$ be an admissible solution to $\sigma_{2}(D^2u)=f(x,u,Du)$, then we have
        \begin{align}\label{eq: est on fi}
            |f_{i}|\leq C\Gamma+C\Delta u,
        \end{align}
        and
        \begin{equation}\label{eq: est on fij}
            -C\Gamma^2(1+\Delta u)^2+\sum_{a=1}^{n}f_{p_{a}}u_{aij}\leq f_{ij}\leq C\Gamma^2(1+\Delta u)^2+\sum_{a=1}^{n}f_{p_{a}}u_{aij},
        \end{equation}
        where $\Gamma=\|u\|_{C^{1}}+1$ and $C$ is a universal constant depending on $n$, $\|f\|_{C^{1,1}}$ ,$\|\frac{1}{f}\|_{L^{\infty}(B_{1})}$.
\end{lemma}
\begin{proof}
    Differentiating $f=f(x,u(x),Du(x))$ directly, we get
    \begin{equation*}
        f_{i}=f_{x_{i}}+f_{z}u_{i}+\sum_{a}f_{p_{a}}u_{ai}
    \end{equation*}
    and
    \begin{equation*}
        \begin{split}
            f_{ij}&=f_{x_{i}x_{j}}+f_{x_{i}z}u_{j}+\sum_{a}f_{x_{i}p_{a}}u_{aj}+f_{x_{j}z}u_{i}+f_{zz}u_{i}u_{j}+\sum_{a}f_{zp_{a}}u_{i}u_{aj}+f_{z}u_{ij}\\
            &\quad +\sum_{a}f_{x_{j}p_{a}}u_{ai}+ \sum_{a}f_{zp_{a}}u_{j}u_{ai}+\sum_{a,b}f_{p_{a}p_{j}}u_{ai}u_{bj}+\sum_{a}f_{p_{a}}u_{aij}.
        \end{split}
    \end{equation*}
    Observe that $(\Delta u)^2=|D^2u|^2+2f> |D^2u|^2$, which implies $|u_{ij}|\leq \Delta u$ for any $i,j=1,2,\cdots n$. Therefore, the estimates \eqref{eq: est on fi} and \eqref{eq: est on fij} follow.
\end{proof}

\begin{proposition}\label{PROP: Almost Jacobi ineq}
    Let $u$ be a smooth admissible solution to \eqref{sigma2 eq} on $B_{1}\subset\bb{R}^n$ and denote $b=\log\Delta u$. Then, in dimension $n= 4$, we have the almost Jacobi inequality:
    \begin{equation}\label{Jacobi ineq}
        \Delta_{F}b\ge \varepsilon |\nabla_{F}b|^2-C\Gamma^2(1+\Delta u)+\sum_{i=1}^{n}f_{p_{i}}b_{i},
    \end{equation}
    for 
    \[ \varepsilon=\dfrac{2}{9}\left(\dfrac{1}{2}+\dfrac{\lambda_{\text{min}}}{\Delta u}\right)>0. \]
    Here $\Gamma=\|u\|_{C^{1}(B_{1})}+1$ and $C$ is a universal constant depending on $n$, $\|f\|_{C^{1,1}(B_{1})}, \|\frac{1}{f}\|_{L^{\infty}(B_{1})}$.
    
    In dimensions $n\ge 5$, the almost Jacobi inequality holds for 
    \[\varepsilon=\dfrac{\sqrt{3n^2+1}-n-1}{3(n-1)}\left(\dfrac{\sqrt{3n^2+1}-n+1}{2n}+\dfrac{\lambda_{\min}}{\Delta u}\right),\]
    provided $u$ satisfies the dynamic semi-convex condition \eqref{eq: dynamic semi-convex condition}:
    \begin{align*}
        \lambda_{\text{min}}(D^2u)\ge -c(n)\Delta u\quad with\quad c(n)=\dfrac{\sqrt{3n^2+1}-n+1}{2n}.
    \end{align*}
    
\end{proposition}

\begin{proof} \emph{Step 1. Expression of Jacobi inequality.} For any fixed point $p$, by choosing an appropriate coordinate such that $D^2u$ is diagonalized at $p$ and $D^2u(p)=\mathrm{diag}\{\lambda_{1},\cdots,\lambda_{n}\}$ with $\lambda_{1}\ge\cdots\ge\lambda_{n}$. Hereafter, all computations are performed at the point $p$. Throughout this proof, $C$ denotes the universal constant depending on $n,\|f\|_{C^{1,1}}$, and $\|\frac{1}{f}\|_{L^{\infty}}$ which may changes line by line. 

Now the coefficients of the linearized operator $\Delta_{F}$ become
\[(F_{ij})=\left(\dfrac{\partial F}{\partial u_{ij}}\right)=\mathrm{diag}\{F_{11},\cdots,F_{nn}\}=\mathrm{diag}\{\Delta u-\lambda_{1},\cdots,\Delta u-\lambda_{n}\}. \]
Computing the derivatives of $b=\log\Delta u$, we obtain
\begin{align}\label{eq: Derivatives of b}
    |\nabla_{F}b|=\sum_{i}\dfrac{F_{ii}(\Delta u_{i})^2}{(\Delta u)^2},\quad \text{and}\quad \Delta_{F}b=\sum_{i,k}\dfrac{F_{ii}u_{iikk}}{\Delta u}-\sum_{i}\dfrac{F_{ii}(\Delta u_{i})^2}{(\Delta u)^2}.
\end{align}
Next, we use the linearized equation to replace the fourth-order term $F_{ii}u_{iikk}$. Differentiating the equation $\sigma_{2}=\frac{1}{2}\left[(\Delta u)^2-|D^2u|^2\right]=f$ with respect to $x_{k}$, we obtain at the point $p$,
\begin{equation}\label{eq: Diff once}
    \Delta_{F}u_{k}=F_{ii}u_{iik}=f_{k}.
\end{equation}
Differentiating \eqref{eq: Diff once} with respect to $x_{k}$ again, we obtain at the point $p$,
\begin{equation}\label{eq: Diff twice}
    \Delta_{F}u_{kk}=\sum_{i}F_{ii}u_{iikk}=f_{kk}+\sum_{i,j}u_{ijk}^2-(\Delta u_{k})^2.
\end{equation}
Put it into \eqref{eq: Derivatives of b}, we obtain
\begin{align*}
    \Delta_{F}b&=\dfrac{1}{\Delta u}\left\{\Delta f +\sum_{i,j,k}u_{ijk}^2-\sum_{i}\left(1+\dfrac{F_{ii}}{\Delta u}\right)(\Delta u_{i})^2 \right\}\\
    &=\dfrac{1}{\Delta u}\left\{\Delta f +6\sum_{i<j<k}u_{ijk}^2+3\sum_{j\neq i}u_{jji}^2+u_{iii}^2-\sum_{i}\left(1+\dfrac{F_{ii}}{\Delta u}\right)(\Delta u_{i})^2 \right\}\\
    &\ge \dfrac{\Delta f}{\Delta u}+\dfrac{1}{\Delta u}\sum_{i}\left\{3\sum_{j\neq i}u_{jji}^2+u_{iii}^2- \left(1+\dfrac{F_{ii}}{\Delta u}\right)(\Delta u_{i})^2 \right\}
\end{align*}
Combining with the gradient term, we obtain
\begin{equation}\label{eq: lower bound for Jacobi}
    \begin{split}
        \Delta_{F}b-\varepsilon|\nabla_{F}b|^2& \ge \dfrac{\Delta f}{\Delta u}+\dfrac{1}{\Delta u}\sum_{i}\left\{3\sum_{j\neq i}u_{jji}^2+u_{iii}^2- \left(1+\delta \dfrac{F_{ii}}{\Delta u}\right)(\Delta u_{i})^2 \right\} \\
        &:= \dfrac{\Delta f}{\Delta u}+\dfrac{1}{\Delta u}\sum_{i}Q_{i}
    \end{split}
\end{equation}
where $\delta:=1+\varepsilon$. In the following proof, all our efforts will be devoted to estimating the lower bound of $Q_{i}$.

For any fixed $i\in\{1,\cdots,n\}$. Denote $t=t_{i}:=(u_{11i},\cdots,u_{nni})$, $e_{i}$ be the standard $i$-th basis of $\bb{R}^n$ and $a=(1,\cdots,1)$. We view $Q_{i}$ as a quadratic form at $t$:
\[ Q_{i}=3|t|^2-2\innerproduct{t}{e_{i}}^2-\left(1+\delta \dfrac{F_{ii}}{\Delta u}\right)\innerproduct{t}{a}^2. \]

\emph{Step 2. Restriction on subspace.} To simplify the notation, we denote $\eta=1+\delta \frac{F_{ii}}{\Delta u}$. Recall the linearized equation \eqref{eq: Diff once}, we have $\innerproduct{DF}{t}=f_{i}$ is controlled, where $DF:=(F_{11},\cdots,F_{nn})$. Therefore, we decompose $e_{i}$ and $a$ into tangential part and normal part with respect to $DF$. Their normal parts are
\begin{align}\label{eq: Def of E and L}
    E:=e_{i}-\dfrac{\innerproduct{e_{i}}{DF}}{|DF|^2}DF \quad \text{and}\quad L:=a-\dfrac{\innerproduct{a}{DF}}{|DF|^2}DF.
\end{align}
Now, using Cauchy-Schwartz inequality, we obtain
\begin{align*}
    Q_{i}&=3|t|^2-2\left( \innerproduct{t}{E}+ \dfrac{\innerproduct{e_{i}}{DF}}{|DF|^2}f_{i} \right)^2-\eta\left(\innerproduct{t}{L}+\dfrac{\innerproduct{a}{DF}}{|DF|^2}f_{i}\right)^2\\
    &\ge 3|t|^2-2(1+\theta)\innerproduct{t}{E}^2-\eta(1+\theta)\innerproduct{t}{L}^2\\
    &\quad -2\left(1+\frac{1}{\theta}\right)\dfrac{\innerproduct{e_{i}}{DF}^2}{|DF|^4}f_{i}^2 -\eta\left(1+\frac{1}{\theta}\right)\dfrac{\innerproduct{a}{DF}^2}{|DF|^4}f_{i}^2,
\end{align*}
where $\theta>0$ to be determined later. For last two terms, notice that $|DF|^2=\sum_{i}F_{ii}^2=\sum_{i}(\Delta u-\lambda_{i})^2=(n-2)(\Delta u)^2+|\lambda|^2$ , then from \eqref{eq: est on fi}, we get
\begin{align*}
    \dfrac{\innerproduct{e_{i}}{DF}^2}{|DF|^4}f_{i}^2 &\leq \dfrac{f_{i}^2}{(n-2)(\Delta u)^2+|\lambda|^2}\leq C\Gamma^2,\\
    \dfrac{\innerproduct{a}{DF}^2}{|DF|^4}f_{i}^2&\leq \dfrac{n f_{i}^2}{(n-2)(\Delta u)^2+|\lambda|^2}\leq C\Gamma^2,
\end{align*}
and from Lemma \ref{LEMMA: est for Fii},
\[ \eta=1+\delta\dfrac{F_{ii}}{\Delta u}\leq 1+(1+\varepsilon)\dfrac{2n-2}{n}\leq C, \quad \text{provided}\  \varepsilon\leq 1.\]
Therefore,
\begin{equation}\label{eq: est for Qi}
    \begin{split}
        Q_{i}&\ge 3|t|^2-2(1+\theta)\innerproduct{t}{E}^2-\eta(1+\theta)\innerproduct{t}{L}^2-C\Gamma^2\left(1+\dfrac{1}{\theta}\right)\\
    &:=\widetilde{Q}_{i}-C\Gamma^2\left(1+\dfrac{1}{\theta}\right).
    \end{split}
\end{equation}
Next, we focus on the quadratic form $\widetilde{Q}_{i}$, we will show that $\widetilde{Q}_{i}\ge0$ for some suitable $\theta$. If $t$ is orthogonal to both $E$ and $L$, then $\widetilde{Q}_{i}=3|t|^2\ge 0$. Thus it suffices to assume that $t$ lies in the subspace $\mathrm{span}\{E,L\}$. The matrix associated to the quadratic form is 
\[\widetilde{Q}_{i}=3I-2(1+\theta)E\otimes E-\eta(1+\theta) L\otimes L.  \]
This symmetric matrix has real eigenvalues. In the non-orthogonal basis $\{E,L\}$, the eigenvector equation is
\begin{align}\label{eq:eigenvector equation}
    \begin{pmatrix}
        3-2(1+\theta)|E|^2 & -2(1+\theta)\innerproduct{E}{L}\\
        -\eta(1+\theta)\innerproduct{E}{L} & 3-\eta(1+\theta)|L|^2
    \end{pmatrix}\begin{pmatrix}
        \alpha\\
        \beta
    \end{pmatrix}=\xi\begin{pmatrix}
        \alpha\\
        \beta
    \end{pmatrix}.
\end{align}
It suffices to show that the above coefficient matrix is nonnegetive. We show that both the trace and determinant are  nonnegative. Recall \eqref{eq: Def of E and L}, we have
\begin{align}\label{eq: inner product between E and L}
    |E|^2=1-\dfrac{F_{ii}^2}{|DF|^2},\quad |L|^2=1-\dfrac{2(n-1)f}{|DF|^2},\quad \text{and}\quad \innerproduct{E}{L}=1-\dfrac{(n-1)F_{ii}\Delta u}{|DF|^2}.
\end{align}

\emph{Step 3. Nonnegativity of trace.} By \eqref{eq:eigenvector equation} and \eqref{eq: inner product between E and L}, we have
\begin{align*}
    \mathrm{tr}&=6-2(1+\theta)|E|^2-\eta(1+\theta)|L|^2\\
    &=6-2(1+\theta)\left(1-\dfrac{F_{ii}^2}{|DF|^2} \right)-(1+\theta)\left(1+\delta\dfrac{F_{ii}}{\Delta u}\right)\left(\dfrac{2(n-1)f}{|DF|^2}\right)\\
    &\ge 3-3\theta-(1+\theta)\delta\dfrac{F_{ii}}{\Delta u}\\
    &\ge 3-3\theta-(1+\theta)\delta \dfrac{2n-2}{n} \ge 0
\end{align*}
provided
\begin{align}\label{eq: trace nonnegative condition}
    \delta\leq \dfrac{3(1-\theta)}{1+\theta}\dfrac{n}{2n-2}.
\end{align}

\emph{Step 4. Nonnegativity of determinant.} There is a slight difference between our calculation here and those in \cite{Shankar-Yuan-Duke-22,Shankar-Yuan-25}. Due to the appearance of small perturbation $\theta$, we have an additional negative term here. 

To simplify the notation, we denote $\beta=1+\theta$. By \eqref{eq:eigenvector equation} and \eqref{eq: inner product between E and L}, we have
\begin{align*}
    \det &=\left(3-2\beta|E|^2\right)\left(3-\eta\beta|L|^2\right)-2\eta\beta^2\innerproduct{E}{L}^2\\
    &=9-9\beta-3\beta\delta\dfrac{F_{ii}}{\Delta u}+3\beta\eta\dfrac{2(n-1)f}{|DF|^2}+6\beta\dfrac{F_{ii}^2}{|DF|^2}\\
    &\quad\  +2\eta\beta^2\left\{ \dfrac{2(n-1)F_{ii}\Delta u}{|DF|^2}-\dfrac{2(n-1)f}{|DF|^2}-\dfrac{nF_{ii}^2}{|DF|^2} \right\}\\
    &=-9\theta +\beta(3-2\beta)\eta\dfrac{2(n-1)f}{|DF|^2}\\
    &\quad\  +\beta\left\{-3\delta\dfrac{F_{ii}}{\Delta u}+(6-2n\beta\eta)\dfrac{F_{ii}^2}{|DF|^2} +4(n-1)\beta\eta\dfrac{F_{ii}\Delta u}{|DF|^2}\right\}. 
\end{align*}
For $\theta\leq 1/4$, we have $3-2\beta\ge1/2$. Note that $\beta\ge 1$, $\eta\ge1$, hence
\begin{equation}\label{eq: estimate for det in dynamic case}
    \begin{split}
        \det &\ge -9\theta+\dfrac{(n-1)f}{|DF|^2}+\beta\dfrac{|DF|^2}{F_{ii}\Delta u}\left\{-3\delta\dfrac{|DF|^2}{(\Delta u)^2}+ (6-2n\beta\eta)\dfrac{F_{ii}}{\Delta u}+4(n-1)\beta\eta\right\}\\
        &:=-9\theta+\dfrac{(n-1)f}{|DF|^2}+\beta \dfrac{|DF|^2}{F_{ii}\Delta u} A.
    \end{split}
\end{equation}

\begin{claim}
    The quantity $A\ge 0$ for $n= 4$ or under the dynamic semi-convexity condition \eqref{eq: dynamic semi-convex condition} in higher dimensions.
\end{claim}

\noindent\emph{Proof of Claim.} Denote $y=F_{ii}/\Delta u$, then $\eta=1+\delta y$. Note that $|DF|^2=(n-1)(\Delta u)^2-2f$, we have
\begin{equation}\label{eq: def on the polynomial q_beta,delta}
    \begin{split}
        A&= -3\delta\left((n-1)-\dfrac{2f}{(\Delta u)^2}\right)+(6-2n\beta(1+\delta y))y+4(n-1)\beta(1+\delta y)\\
        &\ge-2n\beta\delta y^2+(4(n-1)\beta\delta-2n\beta+6)y+(n-1)(4\beta-3\delta)\\
        &:=q_{\delta,\beta}(y).
    \end{split}
\end{equation}
It suffices to show that the quadratic polynomial $q_{\delta,\beta}(y)$ is nonnegative. We split the parameter $\theta$ from $q_{\delta,\theta}(y)$ by writing  $q_{\delta,\theta}(y)=q_{\delta}(y)+\theta R_{\delta}(y)$, where
\begin{align*}
    q_{\delta}(y)&=-2n\delta y^2+(4(n-1)\delta-2n+6)y+(n-1)(4-3\delta),\\
    R_{\delta}(y)&=-2n\delta y^2+(4(n-1)\delta-2n)y+4(n-1).
\end{align*}

\emph{Case 1: $n=4$.} We first estimate $q_{\delta}(y)$.  Following the argument in \cite[P497-498]{Shankar-Yuan-25}, recall $\delta=1+\varepsilon$, we write 
\begin{align*}
    q_{\delta}(y)=\left(-8y^2+10y+3\right)+\varepsilon\left(-8y^2+12y-9\right):=q_{1}(y)+\varepsilon r(y).
\end{align*}
For the remainder $r(y)$:
\[ r(y)=8y\left(\dfrac{3}{2}-y\right)-9\ge -9, \]
where we used $0<y=F_{ii}/\Delta u\leq F_{nn}/\Delta u< (2n-2)/n=3/2 $ from Lemma \ref{LEMMA: est for Fii}.

Next, we solve the equation $q_{1}(y)=0$ and get two roots: $y^{-}=-1/4$ and $y^{+}=3/2$. Hence
\[ q_{1}(y)=8\left(y+\dfrac{1}{4}\right)\left(\dfrac{3}{2}-y\right)\ge 2\left(\dfrac{3}{2}-y\right). \]

For $R_{\delta}(y)$, we employ the similar trick. Decompose
\begin{align*}
    R_{\delta}(y)=\left(-8y^2+4y+12\right)+\varepsilon\left(-8y^2+12y \right):=R_{1}(y)+\varepsilon \widetilde{r}(y)
\end{align*}
The remainder $\widetilde{r}(y)\ge 0$ due to $0<y<3/2$. We solve the equation $R_{1}(y)=0$ and get two roots: $\widetilde{y}^{-}=-1, \widetilde{y}^{+}=3/2$. Hence
\begin{align*}
    R_{1}(y)=8(y+1)\left(\dfrac{3}{2}-y\right)\ge0.
\end{align*}
Therefore,
\begin{align*}
    q_{\delta,\beta}(y)\ge 2\left(\dfrac{3}{2}-y\right)-9\varepsilon\ge0,
\end{align*}
provided 
\begin{align*}
    \varepsilon=\dfrac{2}{9}\left(\dfrac{3}{2}-\dfrac{F_{nn}}{\Delta u}\right)=\dfrac{2}{9}\left(\dfrac{1}{2}+\dfrac{\lambda_{\min}}{\Delta u}\right).
\end{align*}

\vspace{0.1cm}
\emph{Case 2: $n\ge 5$ with dynamic semi-convexity \eqref{eq: dynamic semi-convex condition}.} We write
\begin{align*}
    q_{\delta}(y)&=\left(-2ny^2+(2n+2)y+n-1\right)+\varepsilon\left(-2ny^2+4(n-1)y-3(n-1)\right)\\
    &:=q_{1}(y)+\varepsilon r(y).
\end{align*}
Since $0<y<(2n-2)/n$, we have
\begin{align*}
    r(y)=2ny\left(\dfrac{2n-2}{n}-y\right)-3(n-1)>-3(n-1).
\end{align*}
Next, solving the equation $q_{1}(y)=0$ and get
\begin{align}\label{eq: def on y_n+-}
    y_{n}^{\pm}=\dfrac{n+1\pm\sqrt{3n^2+1}}{2n}, \quad y_{n}^{-}<0<y_{n}^{+}.
\end{align}
Thus, by the dynamic semi-convexity condition \eqref{eq: dynamic semi-convex condition}: $0<y\leq F_{nn}/\Delta u\leq y_{n}^{+}$, we have
\begin{align*}
    q_{1}(y)=2n(y-y_{n}^{-})(y_{n}^{+}-y)\ge -2ny_{n}^{-}(y_{n}^{+}-y).
\end{align*}

Similarly, decompose
\begin{align*}
    R_{\delta}(y)&=\left(-2ny^2+(2n-4)y+4(n-1)\right)+\varepsilon\left(-2ny^2+4(n-1)y\right)\\
    &:= R_{1}(y)+\varepsilon\widetilde{r}(y).
\end{align*}
The remainder $\widetilde{r}(y)\ge0$ due to $0<y<(2n-2)/n$. Solving the equation $R_{1}(y)=0$ and get 
\begin{align*}
    \widetilde{y}_{n}^{\pm}=\dfrac{n-2\pm\sqrt{9n^2-12n+4}}{2n}, \quad \widetilde{y}_{n}^{-}<0<\widetilde{y}_{n}^{+}.
\end{align*}
For $n\ge 5$, there holds $\widetilde{y}_{n}^{+}>y_{n}^{+}$. Thus, under the condition $0<y\leq F_{nn}/\Delta u\leq y_{n}^{+}$, we have
\begin{align*}
    R_{1}(y)=2n\left(y-\widetilde{y}_{n}^{-}\right)\left(\widetilde{y}_{n}^{+}-y\right)\ge0.
\end{align*}
Finally, we conclude that
\begin{align*}
    q_{\delta,\beta}(y)\ge -2ny_{n}^{-}\left(y_{n}^{+}-y\right)-3(n-1)\varepsilon\ge 0,
\end{align*}
provided 
\begin{equation}\label{eq: def of epsilon}
    \begin{split}
        \varepsilon&:=-\dfrac{2ny_{n}^{-}}{3(n-1)}\left(y_{n}^{+}-\dfrac{F_{nn}}{\Delta u}\right)\\
        &= \dfrac{\sqrt{3n^2+1}-n-1}{3(n-1)}\left(\dfrac{\sqrt{3n^2+1}-n+1}{2n}+\dfrac{\lambda_{\min}}{\Delta u}\right).
    \end{split}
\end{equation}
The claim holds. \hfill\#

Now, from \eqref{eq: estimate for det in dynamic case}, we conclude that
\begin{align*}
    \det \ge -9\theta+ \dfrac{(n-1)f}{|DF|^2}\ge0,
\end{align*}
provided
\begin{align*}
    \theta\leq \dfrac{(n-1)f}{9|DF|^2}.
\end{align*}

We need to check the trace condition \eqref{eq: trace nonnegative condition} is also fulfilled. 
Write $\varepsilon=C(n)(c_n+\lambda_{\min}/\Delta u)$. One can check that $C(n)$ is increasing with $\lim_{n\to\infty}C(n)=(\sqrt{3}-1)/3$, and $c_{n}$ is decreasing with $c_{2}=(\sqrt{13}-1)/4$. Note that $\lambda_{\min}\leq \Delta u$, then for $n\ge 2$,
\begin{align*}
    \delta=1+\varepsilon\leq 1+ \dfrac{\sqrt{3}-1}{3}\left(\dfrac{\sqrt{13}-1}{4}+1\right)\approx 1.403
\end{align*}
For $\theta\leq 1/100$, we have
\begin{align*}
    \dfrac{3(1-\theta)}{1+\theta}\dfrac{n}{2n-2}\ge \dfrac{3}{2}\cdot\dfrac{99}{101}\approx 1.470.
\end{align*}
Thus \eqref{eq: trace nonnegative condition} holds if we choose
\begin{align*}
    \theta=\min\left\{\dfrac{1}{100}, \dfrac{(n-1)f}{9|DF|^2}\right\}.
\end{align*}

\emph{Step 4. Conclusion.} Combining \eqref{eq: est on fij} \eqref{eq: lower bound for Jacobi} and \eqref{eq: est for Qi}, we finally conclude that
\begin{align*}
    \Delta_{F}b-\varepsilon|\nabla_{F}b|^2&\ge \dfrac{\Delta f}{\Delta u}-\dfrac{C\Gamma^2}{\Delta u}\left(1+\dfrac{1}{\theta}\right)\\
    &\ge -C\Gamma^2(1+\Delta u)+\sum_{i}f_{p_{i}}b_{i} -C\Gamma^2\left(1+\dfrac{|DF|^2}{\Delta u}\right)\\
    &\ge \sum_{i}f_{p_{i}}b_{i}-C\Gamma^2(1+\Delta u).
\end{align*}
Now the proof of Proposition \ref{PROP: Almost Jacobi ineq} is complete.
\end{proof}

\section{Doubling Inequality}\label{Section: Doubling ineq}
    By the help of the almost Jacobi inequality established in Proposition \ref{PROP: Almost Jacobi ineq}, we now can use Guan-Qiu's test function in \cite[Theorem 4]{Guan-Qiu-Duke-19}\cite[Lemma 4]{Qiu-sigma2Hessian-24} to prove an a priori doubling inequality for the Hessian. 
    
    The proof is almost identical to that of \cite[Proposition 3.1]{Shankar-Yuan-25}, except that here our almost Jacobi inequality includes an additional remainder term. Nevertheless, this does not affect the overall argument. For completeness, we still provide the complete details of the proof here.

\begin{proposition}\label{PROP: Doubling inequality}
    Let $u$ be a smooth admissible solution to \eqref{sigma2 eq} on $B_{3}\subset\bb{R}^n$.
    Assume that $n=4$ or $u$ satisfies the dynamic semi-convex condition \eqref{eq: dynamic semi-convex condition} when $n\ge 5$:
    \begin{align*}
        \lambda_{\text{min}}(D^2u)\ge -c(n)\Delta u\quad with\quad c(n)=\dfrac{\sqrt{3n^2+1}-n+1}{2n}.
    \end{align*}
    
    Then we have the following doubling inequality
    \begin{equation}
        \sup_{B_{2}}\Delta u\leq C\exp\left(C\|u\|_{C^{1}(B_{3})}^{6}\right)\sup_{B_{1}}\Delta u
    \end{equation}
    where $C$ depending only on $n, \|f\|_{C^{1,1}(B_{3})}$ and $\|\frac{1}{f}\|_{L^{\infty}(B_{3})}$.
\end{proposition}

\begin{proof}
    Consider the following test function defined on $B_{3}$:
    \[ P(x)=2\log\rho(x)+\alpha(x\cdot Du-u)+\dfrac{\beta}{2}|Du|^2+ \log\max\{\overline{b},\gamma\}, \]
    where $\rho(x)=9-|x|^2$, $\overline{b}=b-\sup_{B_{1}}b$ for $b=\log\Delta u$ and $\alpha,\beta,\gamma$ are constants to be fixed later. Denote $\Gamma=3\|Du\|_{L^{\infty}(B_{3})}+\|u\|_{L^{\infty}(B_{3})}+1$. Throughout this proof, $C$ denotes the universal constant depending on $n,\|f\|_{C^{1,1}}$ and $\|\frac{1}{f}\|_{L^{\infty}}$ which may changes line by line. Small constants $\alpha,\beta$ and large constant $\gamma$ will be chosen to  also depend only on these.

    Suppose $P$ attains its maximum over $B_{3}$ at $x^*\in B_{3}$. If $|x^*|\leq 1$, we have 
    \begin{equation}\label{eq: upper bound for maxP-1}
        P(x^*)\leq C+\alpha \Gamma+\beta\Gamma^2+\log\gamma.
    \end{equation}
    Thus we assume that $1<|x^*|<3$. If $\overline{b}(x^*)\leq \gamma$, we again get \eqref{eq: upper bound for maxP-1}, hence we also assume that $\overline{b}(x^*)>\gamma$ is sufficiently large.

    Choosing an appropriate coordinate such that $D^2u$ is diagonal at $x^*$ and $D^2u(x^*)=\mathrm{diag}\{\lambda_{1},\cdots,\lambda_{n}\}$ with $\lambda_{1}\ge\cdots\ge\lambda_{n}$. Hereafter, all computations are performed  at $x^*$. Recall that $x^*$ is the maximum point, we have $DP=0, D^2P\leq 0$ at $x^*$. That is, 
    \begin{equation}\label{eq: DP=0}
        \begin{split}
            0=P_{i}&=\dfrac{2\rho_{i}}{\rho}+\alpha  \sum_{k} x_{k}u_{ki}+\beta \sum_{k}u_{k}u_{ki}+\dfrac{\overline{b}_{i}}{\overline{b}} \\
        &=\dfrac{2\rho_{i}}{\rho}+\alpha x_{i}\lambda_{i}+\beta u_{i}\lambda_{i}+\dfrac{\overline{b}_{i}}{\overline{b}}
        \end{split}
    \end{equation}
    for any $i=1,2,\cdots,n$, and
    \begin{align}
        0\ge(P_{ij})=\left(\dfrac{2\rho_{ij}}{\rho}-\dfrac{2\rho_{i}\rho_{j}}{\rho^2}+\alpha u_{ij}+\sum_{k}(\alpha x_{k}u_{kij}+\beta u_{ki}u_{kj}+\beta u_{k}u_{kij})+\dfrac{\overline{b}_{ij}}{\overline{b}}-\dfrac{\overline{b}_{i}\overline{b}_{j}}{\overline{b}^2}\right).
    \end{align}
    Contracting with $(F_{ij})= \mathrm{diag}\{F_{11},\cdots,F_{nn}\}$ where $F_{ii}=\Delta u-\lambda_{i}$, and using 
    \[ \sum_{i}F_{ii}=(n-1)\Delta u,\quad  \sum_{i,j}F_{ij}u_{ij}=2f,\quad \sum_{i,j}F_{ij}u_{ijk}=f_{k}.  \]
    We conclude that
    \begin{equation}\label{eq: negative inequality-1}
    \begin{split}
        0\ge \sum_{i}F_{ii}P_{ii}&=2\dfrac{\sum_{i}F_{ii}\rho_{ii}}{\rho}-2\dfrac{\sum_{i}F_{ii}\rho_{i}^2}{\rho^2}+2\alpha f+\alpha\sum_{k}x_{k}f_{k} \\
          &\quad +\beta\sum_{i}F_{ii}\lambda_{i}^2+\beta \sum_{k}u_{k}f_{k}+\dfrac{\Delta_{F}\overline{b}}{\overline{b}}-\dfrac{|\nabla_{F}\overline{b}|^2}{\overline{b}^2}\\
          &\ge -C\dfrac{\Delta u}{\rho^2}-C(\alpha+\beta\Gamma)|Df|+\beta \sum_{i}F_{ii}\lambda_{i}^2+\dfrac{\Delta_{F}\overline{b}}{\overline{b}}-\dfrac{|\nabla_{F}\overline{b}|^2}{\overline{b}^2}.
    \end{split}
    \end{equation}
    From the almost Jacobi inequality, Proposition \ref{PROP: Almost Jacobi ineq}, we have $\Delta_{F}\overline{b}\ge\varepsilon |\nabla_{F}\overline{b}|^2-C\Gamma^2(1+\Delta u)+\sum_{i}f_{p_{i}}\overline{b}_{i}$ in dimension 4 or under the assumption of dynamic semi-convexity in higher dimensions, where $\varepsilon=C(n)(c_{n}+\lambda_{\min}/\Delta u)$. By substituting this into \eqref{eq: negative inequality-1} and using \eqref{eq: est on fi}, we get
    \begin{equation*}
        \begin{split}
            0\ge -C\dfrac{\Delta u}{\rho^2}-C(\alpha+\beta\Gamma)(\Gamma+\Delta u)+\beta\sum_{i}F_{ii}\lambda_{i}^2-C\dfrac{\Gamma^2}{\gamma}(1+\Delta u)+\sum_{i}f_{p_{i}}\dfrac{\overline{b}_{i}}{\overline{b}}\\
            \quad +\left(c_{n}+\dfrac{\lambda_{min}}{\Delta u}\right)\dfrac{|\nabla_{F}\overline{b}|^2}{\overline{b}}-C\dfrac{|\nabla_{F}\overline{b}|^2}{\overline{b}^2},
        \end{split}
    \end{equation*}
    where we divided $C(n)$ and enlarged $C$. We choose 
    \begin{equation}\label{eq: condition on alpha,beta,gamma}
        \alpha\leq 1,\qquad \beta\Gamma\leq 1,\qquad \gamma\ge\Gamma^2.
    \end{equation}
    Under these conditions, we can estimate
    \[ C(\alpha+\beta\Gamma)(\Gamma+\Delta u)+C\dfrac{\Gamma^2}{\gamma}(1+\Delta u)\leq C\Delta u\leq C\dfrac{\Delta u}{\rho^2}. \]
    Moreover, applying the equation \eqref{eq: DP=0} yields
    \begin{align*}
        \left|\dfrac{\overline{b}_{i}}{\overline{b}}\right|=\left|\dfrac{2\rho_{i}}{\rho}+\alpha x_{i}\lambda_{i}+\beta u_{i}\lambda_{i}\right|\leq \dfrac{C}{\rho}+C(\alpha+\beta\Gamma)\Delta u\leq C\dfrac{\Delta u}{\rho^2}.
    \end{align*}
    Therefore, we obtain
    \begin{equation}\label{eq: negative inequality-2}
        \begin{split}
            0\ge -C\dfrac{\Delta u}{\rho^2}+\beta\sum_{i}F_{ii}\lambda_{i}^2+\left(c_{n}+\dfrac{\lambda_{min}}{\Delta u}\right)\dfrac{|\nabla_{F}\overline{b}|^2}{\overline{b}}-C\dfrac{|\nabla_{F}\overline{b}|^2}{\overline{b}^2},
        \end{split}
    \end{equation}
    If the nonnegative coefficient $c_{n}+\lambda_{\min}/\Delta u$ has a positive lower bound, now the almost Jacobi inequality becomes Jacobi inequality, Qiu's argument in \cite[Lemma 4]{Qiu-sigma2Hessian-24} is still valid. In the other case, we have $\lambda_{\min}/\Delta u$ is close to $-c_{n}$, thus the almost Jacobi inequality degenerates, As a compensation, $\lambda^2_{\min}$ is comparable with $(\Delta u)^2$, this will give us a desired positive term in $\beta\sum_{i}F_{ii}\lambda_{i}^2$.

     \vspace{0.2cm}
     \emph{Case I: $-c_{n}\leq \lambda_{\min}/\Delta u\leq -c_{n}/2$.} Now \eqref{eq: negative inequality-2} becomes
     \begin{equation}\label{eq: negative inequality-3}
         \begin{split}
             0&\ge -C\dfrac{\Delta u}{\rho^2}+\beta F_{nn}^2\lambda_{n}^2-C\dfrac{\sum_{i}F_{ii}\overline{b}_{i}^2}{\overline{b}^2}\ge -C\dfrac{\Delta u}{\rho^2}+\dfrac{c_{n}^2\beta}{4}(\Delta u)^3-C\dfrac{\sum_{i}F_{ii}\overline{b}_{i}^2}{\overline{b}^2},
         \end{split}
     \end{equation}
     where we used $\lambda_{n}^2\ge c_{n}^2(\Delta u)^2/4$ and $F_{nn}=\Delta u-\lambda_{n}\ge \Delta u$. From the equation \eqref{eq: DP=0} again, we have
     \[ \dfrac{\overline{b}_{i}^2}{\overline{b}^2}=\left(\dfrac{2\rho_{i}}{\rho}+\alpha x_{i}\lambda_{i}+\beta u_{i}\lambda_{i}\right)^2\leq \dfrac{C}{\rho^2}+C(\alpha^2+\beta^2\Gamma^2)(\Delta u)^2. \]
     Combining with Lemma \ref{LEMMA: est for Fii}: $F_{ii}\leq C\Delta u$, we deduce from \eqref{eq: negative inequality-3} that
     \begin{equation*}
         \beta (\Delta u)^3\leq C\dfrac{\Delta u}{\rho^2}+C\Gamma+C(\alpha^2+\beta^2\Gamma^2)(\Delta u)^3.
     \end{equation*}
     We assume that $\alpha,\beta$ are sufficiently small so that
     \begin{equation}\label{eq: condition on alpha,beta-1}
         \alpha^2\leq \dfrac{\beta}{3C},\qquad \beta\leq \dfrac{1}{3C\Gamma^2},
     \end{equation}
     then we get $\rho^2\overline{b}\leq \rho^2(\Delta u)^2\leq C(\beta^{-1}+\Gamma)\leq C\beta^{-1}$. This leads an upper bound on $P$: 
     \begin{equation}\label{eq: upper bound for maxP-2}
         P(x^*)\leq C+\log \beta^{-1}.
     \end{equation}

     \vspace{0.2cm}
     \emph{Case II. $\lambda_{\min}/\Delta u\ge-c_{n}/2$.} Now \eqref{eq: negative inequality-2} becomes
     \begin{equation}\label{eq: negative inequality-4}
         \begin{split}
             0&\ge -C\dfrac{\Delta u}{\rho^2}+\beta\sum_{i}F_{ii}\lambda_{i}^2+\left(\dfrac{c_{n}}{2}\overline{b}-C\right)\dfrac{|\nabla_{F}\overline{b}|^2}{\overline{b}^2}\\
         &\ge -C\dfrac{\Delta u}{\rho^2}+\beta\sum_{i}F_{ii}\lambda_{i}^2+\dfrac{c_n \overline{b}}{4}\dfrac{\sum_{i}F_{ii}\overline{b}_{i}^2}{\overline{b}^2},
         \end{split}
     \end{equation}
     provided
     \begin{equation}\label{eq: condition on gamma-1}
      \gamma\ge\dfrac{C}{4c_{n}}.  
     \end{equation}

      We have assumed that $1<|x^*|<3$, so there is at least one $k\in \{1,2,\cdots, n\}$, such that $x_{k}^2>1/n$, where $x^{*}=(x_{1},\cdots,x_{n})$. The proof will be divided into two subcases based on whether $k$ is equal to $1$.

     \vspace{0.1cm}
     
     \emph{Subcase II-1: $x_{1}^2>1/n$.} Suppose 
     \begin{equation}\label{eq: condition on alpha,beta-2}
         \beta\leq \dfrac{\alpha}{2n\Gamma},
     \end{equation}
     from the $DP(x^*)=0$ equation \eqref{eq: DP=0}, we have
     \begin{equation*}
         \begin{split}
             \dfrac{\overline{b}_{1}^2}{\overline{b}^2}&=\left(\dfrac{2\rho_{1}}{\rho}+\alpha x_{1}\lambda_{1}+\beta u_{1}\lambda_{1}\right)^2\ge\dfrac{1}{2}\left(\alpha x_{1}\lambda_{1}+\beta u_{1}\lambda_{1}\right)^2-\dfrac{C}{\rho^2}\\
             &\ge \dfrac{\alpha^2}{8n^2}\lambda_{1}^2-\dfrac{C}{\rho^2}\ge \dfrac{\alpha^2}{8n^4}(\Delta u)^2-\dfrac{C}{\rho^2}.
         \end{split}
     \end{equation*}
    If $\alpha^2(\Delta u)^2/n^4\leq 16C/\rho^2$, we have $\rho^2\overline{b}\leq \rho^2(\Delta u)^2\leq C\alpha^{-2}$. This also leads a desired upper bound on $P$:
    \begin{align}\label{eq: upper bound for maxP-3}
        P(x^*)\leq C+\log \alpha^{-2}.
    \end{align}
    In the alternative case, there holds
    \begin{equation*}
        \dfrac{\overline{b}_{1}^2}{\overline{b}^2}\ge \dfrac{\alpha^2}{16n^4}(\Delta u)^2.
    \end{equation*}
    Substituting into \eqref{eq: negative inequality-4}, we obtain
    \begin{equation*}
        \begin{split}
            0&\ge-C\dfrac{\Delta u}{\rho^2}+\dfrac{c_{n}\overline{b}}{4}F_{11}\dfrac{\overline{b}_{1}^2}{\overline{b}^2}\ge -C\dfrac{\Delta u}{\rho^2}+\dfrac{c_{n}\alpha^2}{64n^4}\overline{b}F_{11}(\Delta u)^2
        \end{split}
    \end{equation*}
    Note that $F_{11}\Delta u\ge f$ from Lemma \ref{LEMMA: est for Fii}. By dividing a positive constant and enlarging $C$ , we deduce that $\rho^2\overline{b}\leq C\alpha^{-2}$. This still implies 
    \begin{equation}\label{eq: upper bound for maxP-4}
        P(x^*)\leq C+\log\alpha^{-2}.
    \end{equation}

    \vspace{0.1cm}
    \emph{Subcase II-2: $x_{k}^2\ge1/n$ for some $k\ge2$.} From the $DP(x^*)=0$ equation \eqref{eq: DP=0}, we deduce that
    \begin{equation*}
        \dfrac{\overline{b}_{k}^2}{\overline{b}^2}=\left(\dfrac{2\rho_{k}}{\rho}+\alpha x_{k}\lambda_{k}+\beta u_{k}\lambda_{k}\right)^2\ge \dfrac{2\rho_{k}^2}{\rho^2}-\alpha^2 x_{k}^2\lambda_{k}^2-\beta^2 u_{k}^2\lambda_{k}^2\ge \dfrac{8}{n\rho^2}-C(\alpha^2+\beta^2\Gamma^2)\lambda_{k}^2,
    \end{equation*}
    where we used $\rho_{k}^2=4x_{k}^2\ge 4/n$. By substituting this into \eqref{eq: negative inequality-4}, and using the estimate: $F_{kk}\ge (1-1/\sqrt{2})\Delta u\ge \Delta u/10$,  we obtain
    \begin{equation}\label{eq: negative inequality-5}
        \begin{split}
            0&\ge -C\dfrac{\Delta u}{\rho}+\beta F_{kk}\lambda_{k}^2+\dfrac{c_{n}}{4}\overline{b}F_{kk}\dfrac{\overline{b}_{k}^2}{\overline{b}^2}\\
            &\ge -C\dfrac{\Delta u}{\rho}+\dfrac{\beta}{10}\Delta u\lambda_{k}^2+\dfrac{c_{n}}{40}\gamma\Delta u\left(\dfrac{8}{n\rho^2}-C(\alpha^2+\beta^2\Gamma^2)\lambda_{k}^2\right)\\
            &\ge \dfrac{\Delta u}{\rho^2}\left(\dfrac{c_{n}}{5n}\gamma-C\right)+\Delta u\lambda_{k}^2\left(\dfrac{\beta}{10}-C\alpha^2\gamma-C\beta^2\Gamma^2\gamma\right).
        \end{split}
    \end{equation}
    First, we choose $\gamma$ sufficiently large so that
    \begin{align}\label{eq: condition on gamma-2}
        \gamma\ge \dfrac{10nC}{c_{n}}.
    \end{align}
    Then we choose $\alpha,\beta$ sufficiently small so that
    \begin{align}\label{eq: condition on alpha,beta-3}
        \alpha^2\leq \dfrac{\beta}{30C\gamma},\qquad \beta\leq \dfrac{1}{30C\Gamma^2\gamma}.
    \end{align}
    Now \eqref{eq: negative inequality-5} leads a contradiction.

    Finally, for $C$ large enough, we take
    \[ \gamma=\dfrac{10n C}{c_{n}}\Gamma^2,\qquad \alpha=\dfrac{1}{60nC\Gamma\gamma},\qquad \beta=\dfrac{1}{120n^2C\Gamma^2\gamma}.  \]
    One can check that previous conditions on $\alpha,\beta,\gamma$: \eqref{eq: condition on alpha,beta,gamma} \eqref{eq: condition on alpha,beta-1} \eqref{eq: condition on gamma-1} \eqref{eq: condition on alpha,beta-2} \eqref{eq: condition on gamma-2} and \eqref{eq: condition on alpha,beta-3} are  satisfied. Hence, from the upper bound on $P$: \eqref{eq: upper bound for maxP-1} \eqref{eq: upper bound for maxP-2} \eqref{eq: upper bound for maxP-3} and \eqref{eq: upper bound for maxP-4}, we conclude that
    \begin{equation*}
       \max_{B_{2}}P\leq P(x^*)\leq C+\max\{\log\alpha^{-2},\log\beta^{-1},\log\gamma\} \leq C+\log \Gamma^6.
    \end{equation*}
    By the definition of $P$ and $\overline{b}$, we finally obtain the doubling inequality:
    \[ \dfrac{\max_{B_{2}}\Delta u}{\max_{B_{1}}\Delta u}\leq \exp(C\Gamma^6). \]
    Now the proof is complete. 
\end{proof}

\section{Alexandrov Regularity}\label{Section: Alexandrov Regularity}

In \cite[Section 4]{Shankar-Yuan-25}, they proved a Alexandrov type regularity theorem for viscosity solutions to $\sigma_{2}=1$. As mentioned in their paper, the main two ingredients in the proof of the classic Alexandrov theorem for convex functions are ``$W^{2,1}$ estimate'' and the gradient (or Lipschitz) estimate, see \cite[Theorem 6.9]{Evans-Gariepy}. For general $k$-convex functions, there is no gradient estimate, but only H\"older estimates for $k>n/2$ \cite{Trudinger-Wang-Annals-99}. This is sufficient to prove the Alexandrov regularity, see \cite{Chaudhuri-Trudinger-05}. For 2-convex solutions to the sigma-2 equation $\sigma_{2}=1$, thanks to \cite{Trudinger-CPDE-97,Trudinger-Wang-Annals-99, Chou-Wang-01}, we have the desired gradient estimate. Thus the method of \cite{Evans-Gariepy,Chaudhuri-Trudinger-05} can go through verbatim. 

For general right-hand side terms $f(x,u,Du)$, the a priori gradient estimates in \cite{Trudinger-CPDE-97,Chou-Wang-01} may not be directly applicable. This limitation arises since the solvability of Dirichlet problems becomes a nontrivial issue in such cases, where their a priori estimates may not be extended to viscosity solutions through smooth approximation procedures. To address this challenge, we employ Labutin's potential estimate \cite{Labutin-Duke-02} to establish Hölder estimates for viscosity solutions. This approach ultimately enables us to prove the Alexandrov regularity theorem for Lipschitz solutions of the $\sigma_{2}=f$ equation with $L^{\infty}$ right-hand side terms. This is sufficient for our final compactness argument. 

\begin{proposition}\label{PROP: Alexandrov type thm}
    Let $u$ be a 2-convex Lipschitz viscosity solution of the sigma-2  equation $\sigma_{2}(D^2u)=f$ on $B_{1}\subset\bb{R}^n$. Suppose $n\ge 4$ and $f$ is bounded, then $u$ is twice differentiable almost everywhere in $B_{1}$. That is, for almost every $x\in B_{1}$, there is a quadratic polynomial $Q$ such that 
    \[ |u(y)-Q(y)|=o(|y-x|^2),\quad as\ y\to x. \]
\end{proposition}

\begin{remark}
    For $n\leq 3$, the almost everywhere twice differentiability of $2$-convex functions was already proved by \cite{Chaudhuri-Trudinger-05}. 
\end{remark}

\subsection{Weighted H\"older Estimates} The key step in the proof of Proposition \ref{PROP: Alexandrov type thm} is the following H\"{o}lder estimates.

\begin{proposition}\label{PROP: weighted Holder estimate}
    Let $u$ be a 2-convex viscosity solution to $\sigma_{2}(D^2u)=f$ in $B_{R}\subset\bb{R}^n$. Suppose $n\ge 4$ and $f\in L^{\infty}(B_{R})$, then there exist universal constants $\gamma=\gamma(n)\in (0,1)$ and $C=C\left(n,\|f\|_{L^{\infty}(B_{R})}\right)>0$, such that
    \begin{align}\label{eq: weighted gradient estimate}
        \sup_{\substack{x,y\in B_{R}\\ x\neq y}}d_{x,y}^{n+\gamma}\dfrac{|u(x)-u(y)|}{|x-y|^{\gamma}}\leq C\left( \int_{B_{R}}|u|\,\mathrm{d}x+R^{n+2} \right),
    \end{align}
    where $d_{x,y}=\min\{d_{x},d_{y}\}$, and $d_{x}=\mathrm{dist}(x,\partial B_{R})=R-|x|$.
\end{proposition}
 
 We adopt notations consistent with \cite[P585]{Trudinger-Wang-Annals-99}. For $\gamma\in(0,1]$, denote
    \begin{align*}
        |u|_{0;R}^{(n)}&=\sup_{x\in B_{R}}d_{x}^{n}|u(x)|,\\
        [u]_{0,\gamma;R}^{(n)}&= \sup_{\substack{x,y\in B_{R}\\ x\neq y}}d_{x,y}^{n+\gamma}\dfrac{|u(x)-u(y)|}{|x-y|^{\gamma}}.
    \end{align*}

There is an interpolation inequality between these two norms, see \cite[Lemma 2.6, P585]{Trudinger-Wang-Annals-99}.
\begin{lemma}\label{LEMMA: interpolation ineq}
    For any $\varepsilon>0$ and $u\in C^{0}(B_{R})\cap L^{1}(B_{R})$, there holds
    \begin{align}\label{eq: interpolation ineq}
        |u|_{0;R}^{(n)}\leq \varepsilon^{\gamma}[u]_{0,\gamma;R}^{(n)}+C(n)\varepsilon^{-n}\int_{B_{R}}|u|.
    \end{align}
\end{lemma}

We now introduce the Wolff potential for 2-convex functions. Indeed, it can be defined for $k$-convex functions with $k\leq n/2$.

\begin{definition}[Wolff Potential]
    Let $n\ge 4$, and let $u$ be a 2-convex function on a domain $\Omega\subset\bb{R}^n$. For any $B_{r}(x)\subset \Omega$, define the Wolff potential of $u$ by
    \[ W(x,r)=\int_{0}^{r}\left( \dfrac{\mu(B_{t}(x))}{t^{n-4}} \right)^{\frac{1}{2}}\dfrac{\mathrm{d}t}{t}, \]
    where $\mu=\sigma_{2}(D^2u)$ is the 2-Hessian measure of $u$ which is introduced by Trudinger-Wang \cite{Trudinger-Wang-Annals-99}.
\end{definition}

In \cite[Theorem 2.1, P6]{Labutin-Duke-02}, Labutin proved the following estimates of 2-convex functions in terms of the Wolff potential. It should be mentioned that Labutin's result holds for $k$-convex functions with $0<k\leq n/2$. However, we specifically focus on and present the case of $k=2$ and $n\ge 4$ in this work. 

\begin{lemma}\label{LEMMA: est in terms of Wolff potential}
    Let $u$ be a nonnegative 2-convex functions on $B_{3R}\subset\bb{R}^n$ with $n\ge 4$. Then 
    \[ C_{1}W\left(0,\dfrac{R}{8}\right)\leq u(0)\leq C_{2}W(0,2R)+C_{3}\inf_{B_{R}}u, \]
    where $C_{1},C_{2}$ and $C_{3}$ depend only on $n$.
\end{lemma}

From this estimate, we derive the Harnack-type inequality in the following form.
\begin{lemma}[Harnack estimates]\label{LEMMA: Harnack estimate}
    Let $u$ be a nonnegative 2-convex solution to $\sigma_{2}(D^2)=f$ on $B_{R}\subset \bb{R}^n$. Suppose $n\ge 4$ and $f\in L^{\infty}(B_{R})$, then for any $0<r<R/10$, we have
    \begin{align}\label{eq: Harnack estimate}
        \sup_{B_{r}}u\leq C_{1}\inf_{B_{r}}u+C_{2}r^2,
    \end{align}
    where $C_{1}$ depends only on $n$ and $C_{2}$ depends on $n$ and $\|f\|_{L^{\infty}(B_{R})}$.
\end{lemma}
\begin{proof}
    For any $x\in B_{r}$, applying Lemma \ref{LEMMA: est in terms of Wolff potential} on $B_{6r}(x)$, we obtain
    \begin{align}\label{eq: 4.4}
        u(x)\leq C(n)\left(W(x,4r)+\inf_{B_{2r}(x)}u\right)\leq C(n)\left( W(x,4r)+\inf_{B_{r}}u \right).
    \end{align}
    To estimate the Wolff potential, note that the 2-Hessian measure of $u$ is $\mu=f\,\mathrm{d}x$ now. This yields $\mu(B_{t}(x))\leq C(n)\|f\|_{L^\infty}t^n$. Therefore, 
    \begin{align}\label{eq: est on W(x,4r)}
        W(x,4r)\leq\int_{0}^{4r}\left(\dfrac{\mu(B_{t}(x))}{t^{n-4}}\right)^{\frac{1}{2}}\dfrac{\mathrm{d}t}{t}\leq C(n)\|f\|_{L^{\infty}}^{1/2}r^2.
    \end{align}
    Combining \eqref{eq: 4.4} and \eqref{eq: est on W(x,4r)}, the result follows.
\end{proof}

Now we are ready to prove Proposition \ref{PROP: weighted Holder estimate}.

\vspace{0.2cm}

\noindent\emph{Proof of Proposition \ref{PROP: weighted Holder estimate}}.  We fix any $x\in B_{R}$ and let $d=d_{x}=R-|x|$. For any $0<r<d$, denote the oscillation of $u$ over $B_{r}(x)$ by
\[ \omega_{r}:=\mathrm{osc}_{B_{r}(x)}u=\sup_{B_{r}(x)}u-\inf_{B_{r}(x)}u. \]

\emph{Step 1. Oscillation Decay.} We first prove the oscillation decay property: 
\begin{align}\label{eq: oscillation decay}
     \omega_{r}\leq \theta \omega_{10r}+ Cr^2, \qquad \text{for any}\  0<r<\frac{d}{10}, 
\end{align}
where $\theta=\theta(n)\in (0,1)$ and $C=C\left(n,\|f\|_{L^{\infty}}\right)$.

This follows from the Harnack-type inequality easily. Let $M_{r}=\sup_{B_{r}(x)}u$ and $m_{r}=\inf_{B_{r}(x)}u$. Applying the Harnack-type inequality for $M_{10r}-u$ and $u-m_{10r}$ on $B_{10r}(x)$, we derive
\begin{align*}
    M_{10r}-m_{r}&\leq C_{1}(M_{10r}-M_{r})+C_{2}r^2,\\
    M_{r}-m_{10r}&\leq C_{1}(m_{r}-m_{10r})+C_{2}r^2.
\end{align*}
Adding these two inequalities yields \eqref{eq: oscillation decay} for $\theta=(C_1-1)/(C_{1}+1)\in (0,1)$.

\emph{Step 2. Iteration argument.} For any $0<r<d/10$, choose a positive integer $k$ such that $d/10^{k+1}<r\leq d/10^k$. Set $r_{i}=d/10^{i}$ for $i=1,2,\cdots,k$. Iterating the oscillation decay property \eqref{eq: oscillation decay}, 
\begin{align*}
    \omega_{r}\leq \omega_{r_{k}}&\leq \theta \omega_{r_{k-1}}+Cr_{k}^2\\
    &\leq \theta(\theta \omega_{r_{k-2}}+Cr_{k-1}^2 )+Cr_{k}^2\\
    &\leq \cdots\leq \theta^{k-1}\omega_{r_{1}}+C\sum_{i=0}^{k-2}\theta^{i}r_{k-i}^2. 
\end{align*}
Thus, there exists $\gamma=\gamma(n)\in (0,1)$, such that
\[ \omega_{r}\leq C\left(\dfrac{r}{d}\right)^{\gamma}\omega_{\frac{d}{10}}+Cd^2,\qquad \text{for any}\ 0<r<\frac{d}{10}. \]
By a well-known iteration lemma of monotone functions, see \cite[Chapter 6]{GT-book} and \cite[Chapter 3]{HanLin-Book}, we can improve the above estimate to
\begin{align}\label{eq: improved oscillation decay}
    \omega_{r}\leq C\left(\dfrac{\omega_{d/10}}{d^{\gamma}}r^{\gamma}+r^2\right),\qquad \text{for any}\ 0<r<\frac{d}{10}.
\end{align}

\emph{Step 3. Estimate the difference quotient.} For any $y\in B_{R}$, consider two cases:

\emph{Case 1. $r=|x-y|<d/20$.} By \eqref{eq: improved oscillation decay}, we obtain 
\begin{align*}
    d_{x,y}^{n+\gamma}\dfrac{|u(x)-u(y)|}{|x-y|^{\gamma}}&\leq d^{n+\gamma}\dfrac{\omega_{r}}{r^{\gamma}}\leq \dfrac{d^{n+\gamma}}{r^\gamma}\cdot C\left(\dfrac{\omega_{d/10}}{d^\gamma}r^{\gamma}+r^2\right)\\
    &\leq Cd^n \sup_{B_{d/10}(x)}|u|+Cd^{n+\gamma}r^{2-\gamma}\leq C|u|_{0;R}^{(n)}+CR^{n+2}.
\end{align*}

\emph{Case 2. $|x-y|\ge d/20$.}
\begin{align*}
    d_{x,y}^{n+\gamma}\dfrac{|u(x)-u(y)|}{|x-y|^{\gamma}}&\leq Cd_{x,y}^{n+\gamma}\dfrac{|u(x)|+|u(y)|}{d^{\gamma}}\leq C(d_{x}^n|u(x)|+d_{y}^n|u(y)|)\leq C|u|_{0;R}^{(n)}.
\end{align*}
Combining these two cases, we obtain
\[ [u]_{0,\gamma;R}^{(n)}\leq C\left(|u|_{0;R}^{(n)}+R^{n
+2}\right). \]
Finally, applying the interpolation inequality \eqref{eq: interpolation ineq} with $\varepsilon^{\gamma}=1/2C$, we conclude that
\[ [u]_{0,\gamma;R}^{(n)}\leq C\left(\int_{B_{R}}|u|+R^{n
+2}\right), \]
where $C$ depends only on $n$ and $\|f\|_{L^{\infty}(B_{R})}$.\hfill\qed

\subsection{Proof of Proposition \ref{PROP: Alexandrov type thm}} We follow the arguments in \cite{Evans-Gariepy, Chaudhuri-Trudinger-05} to prove Proposition \ref{PROP: Alexandrov type thm}.

\vspace{0.2cm}

\noindent\emph{Proof of Proposition \ref{PROP: Alexandrov type thm}.} \emph{Step 1. Hessian as Radon measures.}   For 2-convex function $u$, Chaudhuri-Trudinger \cite[Theorem 2.4]{Chaudhuri-Trudinger-05} proved in the distribution sense that the Hessian of $u$ can be interpreted as a matrix-valued Radon measure $[D^2u]=[\mu^{ij}]$ with $\mu^{ij}=\mu^{ji}$:
\[ \int u\varphi_{ij}=\int \varphi \,\mathrm{d}\mu^{ij}\qquad \mathrm{for\ any}\ \varphi\in C_{c}^{\infty}(B_{1}). \]

By Lebesgue-Radon-Nikodym decomposition, we write  $\mu^{ij}=u^{ij}\,\mathrm{d}x+\mu^{ij}_{s}$, where $\mathrm{d}x$ denotes the $n$-dimensional Lebesgue measure, $u^{ij}\in L^{1}_{\mathrm{loc}}$ denotes the absolutely continuous part with respect to $\mathrm{d}x$, and $\mu^{ij}_{s}$ denotes the singular part. Write $[D^2u]=D^2u\, \mathrm{d}x+[D^2u]_{s}$, where $D^2u=(u^{ij})$ and $[D^2u]_{s}=[\mu^{ij}_{s}]$. For almost every $x\in B_{1}$, there hold
\begin{align}
    &\lim_{r\to 0}\fint_{B_{r}(x)}|D^2u(y)-D^2u(x)|\,\mathrm{d}y=0, \label{eq: 6.1} \\
    &\lim_{r\to 0}\dfrac{1}{r^n}\|[D^2u]_{s}\|(B_{r}(x))=0. \label{eq: 6.2}
\end{align}
Here $\|[D^2u]_{s}\|$ denotes the total variation of $[D^2u]_{s}$. Since we assumed that $u$ is Lipschitz,  Rademacher's theorem tells us that $u$ is differentiable almost everywhere in $B_{1}$. In particular, for almost every $x\in B_{1}$, we also know that
\begin{equation}\label{eq: 6.3}
    \lim_{r\to 0}\fint_{B_{r}(x)}|Du(y)-Du(x)|\,\mathrm{d}y=0.
\end{equation}

Fix any $x$ such that \eqref{eq: 6.1} \eqref{eq: 6.2} and \eqref{eq: 6.3} hold, we will show that $h(y)=o(|y-x|^2)$ as $y\to x$, where
\[ h(y)=u(y)-u(x)-Du(x)\cdot(y-x)-\dfrac{1}{2}(y-x)^{T}D^2u(x)(y-x). \]

\emph{Step 2. Approximation in $L^1$ sense.} Following verbatim Steps 2–4 in the proof of Theorem 6.9 from \cite[P274-275]{Evans-Gariepy}, we can conclude that
\[ \fint_{B_{r}(x)}|h(y)|\,\mathrm{d}y=o(r^2),\qquad \mathrm{as}\ r\to0. \]

\emph{Step 3. H\"older estimate.} Fix any $r$ satisfying $0<2r<1-|x|$. Consider $g(y)=u(y)-u(x)-Du(x)\cdot(y-x)$. It is a 2-convex viscosity solution to the sigma-2 equation $\sigma_{2}(D^2g)=f$ on $B_{2r}(x)$. By Proposition \ref{PROP: weighted Holder estimate}, we have
\begin{equation}
    \begin{split}
        r^{n+\gamma}\sup_{\substack{y,z\in B_{r}(x)\\ y\neq z}}\dfrac{|g(y)-g(z)|}{|y-z|^{\gamma}}&\leq \sup_{\substack{y,z\in B_{2r}(x)\\ y\neq z}}d_{y,z}^{n+\gamma}\dfrac{|g(y)-g(z)|}{|y-z|^{\gamma}}\\
        &\leq C\left(\int_{B_{2r}(x)}|g(y)|\,\mathrm{d}y+r^{n+2}\right)\\
        &\leq C\int_{B_{2r}(x)}|h(y)|\,\mathrm{d}y+ Cr^{n+2}
    \end{split}
\end{equation}
where $C$ depends on $n, \|f\|_{L^{\infty}(B_{1})}$ and $|D^2u(x)|$. We also notice that
\[ (y-x)^{T}D^2u(x)(y-x)-(z-x)^{T}D^2u(x)(z-x)=(y+z-2x)^{T}D^2u(x)(y-z). \]
Therefore, we obtain the following H\"older estimate:
\begin{equation}\label{eq:6.5}
    \begin{split}
        \sup_{\substack{y,z\in B_{r}(x)\\ y\neq z}}\dfrac{|h(y)-h(z)|}{|y-z|^{\gamma}}&\leq \sup_{\substack{y,z\in B_{r}(x)\\ y\neq z}}\dfrac{|g(y)-g(z)|}{|y-z|^{\gamma}}+Cr^{2-\gamma}\\
        &\leq \dfrac{C}{r^{\gamma}}\fint_{B_{2r}(x)}|h(y)|\,\mathrm{d}y+Cr^{2-\gamma}.
    \end{split}
\end{equation}

\emph{Step 4. Improve $L^{1}$ approximation to $L^{\infty}$.} For any small $\varepsilon>0$, we will find $r_{0}$ small, such that
\[ \dfrac{1}{r^2}\sup_{B_{r/2}(x)}|h|\leq 2\varepsilon \qquad \mathrm{for}\ r<r_{0}.   \]
Take $0<\eta<1/2$ small to be fixed later. By Step 2, we have
\begin{equation}\label{eq: 6.6}
    \mathcal{L}^{n}\left(\{z\in B_{r}(x): |h(z)|\ge\varepsilon r^2\}\right)\leq \dfrac{1}{\varepsilon r^2}\fint_{B_{r}(x)}|h(z)|\,\mathrm{d}z\leq \dfrac{o(r^{2})}{\varepsilon r^2}<\dfrac{1}{2}\eta^n \mathcal{L}^n(B_{r}),
\end{equation}
provided $r<r_{0}=r_{0}(n,\eta,\varepsilon)$ small. We claim that for any $y\in B_{r/2}(x)$, there exists $z\in B_{\eta r}(y) $ such that $|h(z)|<\varepsilon r^2$. Otherwise, we have $B_{\eta r}(y)\subset \{|h|\ge \varepsilon r^2\}\cap B_{r}(x)$, this implies
\[ \eta^n\mathcal{L}^{n}(B_{r})\leq \mathcal{L}^{n}\left(\{z\in B_{r}(x): |h(z)|\ge\varepsilon r^2\}\right)<\dfrac{1}{2}\eta^n \mathcal{L}^{n}(B_{r}). \]
It contradicts with \eqref{eq: 6.6}.  Therefore,
\begin{align*}
    |h(y)|\leq |h(z)|+(\eta r)^{\gamma} \dfrac{|h(y)-h(z)|}{|y-z|^{\gamma}}&\leq \varepsilon r^2+(\eta r)^{\gamma}\left(\dfrac{C}{r^{\gamma}}\fint_{B_{2r}(x)}|h|+Cr^{2-\gamma}\right)\\
    &\leq \varepsilon r^2+\eta^{\gamma} o(r^2)+ C\eta^{\gamma} r^2. 
\end{align*}
We choose $\eta^{\gamma}=\varepsilon/2C$ and reduce $r_{0}$ if needed, then we finally conclude that
\begin{align*}
    |h(y)|\leq  2\varepsilon r^2\qquad \mathrm{for}\ r<r_{0}=r_{0}(n,\varepsilon).
\end{align*}
Since $y\in B_{r/2}(x)$ is arbitrary, we finish the proof.\hfill\qed

\section{A Generalization of Savin's Small Perturbation Theorem}\label{Section: Genaralized Savin's thm}

In \cite{lian-zhang} and \cite{Fan}, we can generalize Savin's small perturbation theorem \cite{Savin-07} to fully nonlinear elliptic equations with non-homogeneous terms. Consider the fully nonlinear equation of general form: 
 \[ F(D^2u,Du,u,x)=f(x)\quad \mathrm{on}\ B_{1}, \]
    where $F:\mathcal{S}^{n}\times \bb{R}^n\times\bb{R}\times B_{1}\to\bb{R}$ be a function defined for pairs $(M,p,z,x)$ satisfying the following hypotheses:

    $\mathrm{H}1)$ $F(\cdot,p,z,x)$ is elliptic, i.e. 
        \[ F(M+N,p,z,x)\ge F(M,p,z,x) \quad \text{for}\ M, N\in\mathcal{S}^{n}\ \text{and}\ N\ge0. \]

    $\mathrm{H}2)$ $F(\cdot,p,z,x)$ is uniformly elliptic in a $\rho\,$-neighborhood of the origin in $\mathcal{S}^n$. That is, there exist constants $\Lambda>\lambda>0$, such that for any $\|M\|, \|N\|, |p|, |z|\leq \rho$ with $N\ge0$,
        \[ \Lambda\|N\|\ge F(M+N,p,z,x)-F(M,p,z,x)\ge \lambda\|N\|. \]

    $\mathrm{H}3)$ $0$ is a solution, i.e. $F(0,0,0,x)\equiv 0$. Moreover, $F$ satisfies the structure condition: for any $\|M\|, |p|,|q|,|z|,|s|\leq \rho$ and $x\in B_{1}$, there holds
        \[  |F(M,p,z,x)-F(M,q,s,x)|\leq b_{0}|p-q|+c_{0}|z-s|.  \]

    $\mathrm{H}4)$ $F\in C^{1}$ and its derivatives $D_{M}F$ is uniformly continuous in the $\rho\,$-neighborhood of $\{(0,0,0,x):x\in B_{1}\}$ with modulus of continuity $\omega_{F}$. 
    
    The generalized small perturbation theorem is stated as follows 
    \begin{theorem}\label{THM: Generalized Savin's thm}
        Let $F$ satisfy $\rm{H}1)\sim\rm{H}4)$, and let $u\in C(B_{1})$ be a viscosity solution to 
        \[ F(D^2u, Du, u, x)=f(x)\qquad \mathrm{on}\ B_{1}. \]
        For any $\alpha\in (0,1)$, there exist constants $\delta,C>0$, depending only on $n,\alpha, \rho,\lambda,\Lambda$, $b_{0}$, $c_{0}$ and $\omega_{F}$, such that if
        \[ |F(M,p,z,x)-F(M,p,z,x')|\leq \delta |x-x'|^{\alpha}\quad \text{for all}\ \|M\|,|p|,|z|\leq \rho\ \text{and}\ x,x'\in B_1, \]
        and        
        \[ \|u\|_{L^{\infty}(B_{1})}\leq \delta, \qquad \|f\|_{C^{0,\alpha}(B_{1})}\leq \delta, \]
         then $u\in C^{2,\alpha}(B_{1/2})$ with 
        \[\|u\|_{C^{2,\alpha}(B_{1/2})}\leq C.\]
    \end{theorem}

    \section{Proof of Main Theorems}\label{Section: Proof of Main Theorems}

    \emph{Proof of Theorem \ref{THM: main thm} and Theorem \ref{THM: est for dynamic semi-convex solus}.} After scaling $16u(x/4)$, we will prove these theorems on  $B_{4}$.
    
    \emph{Step 1.} If the Hessian estimate were false, there would exist sequences of $\{u_{k}\}$ and $\{f_{k}\}$ such that:

    $i)$ $u_{k}$ is a $2$-convex smooth solution to  $\sigma_{2}(D^2u_{k})=f_{k}(x,u_{k},Du_{k})\ \mathrm{on}\ B_{4}$ in either dimension $n=4$ or in higher dimensions $n\ge 5$ with $u_{k}$ also satisfies the dynamic semi-convex condition \eqref{eq: dynamic semi-convex condition};

    $ii)$ $\|u_{k}\|_{C^{1}(B_{4})}+\|f_{k}\|_{C^{1,1}(B_{4}\times \bb{R}\times \bb{R}^n)}\leq A $ and $\inf f_{k}\ge f_{0}>0$;

    $iii)$ but the Hessian $|D^2u_{k}(0)|\to\infty$ as $k\to\infty$.

    By ii), applying Arzela-Ascoli, up to subsequence, $u_{k}$ uniformly converges $u$ on $B_{3}$ for some Lipschitz $u$, and $f_{k}$ converges to $f$ in $C^{1,\alpha}(B_{3})$ for any $0<\alpha<1$ for some $f\in C^{1,1}$. By the closedness of viscosity solutions (see \cite{CCbook}), $u$ is a viscosity solution to $\sigma_{2}(D^2u)=f(x,u,Du)$ on $B_{3}$. Since the limit solution $u$ is Lipschitz, we know that $f(x,u,Du)\in L^{\infty} $. By the Alexandrov regularity proposition \ref{PROP: Alexandrov type thm}, $u$ is twice differentiable almost everywhere. We fix a such twice differentiable point $y\in B_{1/3}$ and let $Q(x)$ be the quadratic polynomial such that $|u(x)-Q(x)|=o(|x-y|^2)$ near $y$. Notice that $\sigma_{2}(D^2Q)=f(y,Q(y),DQ(y))=f(y,u(y),Du(y))$.

    \emph{Step 2.} Set $v_{k}=u_{k}-Q$. For small $r>0$, consider the rescaled functions
    \[ \widetilde{v}_{k}(x)=\dfrac{1}{r^2}v_{k}(rx+y)\quad for\ x\in B_{1}. \]
    Then 
    \begin{align*}
        \sigma_{2}(D^2\widetilde{v}_{k}(x)+D^2Q)&=\sigma_{2}(D^2u_{k}(rx+y))=f_{k}(rx+y,u_{k}(rx+y),Du_{k}(rx+y))\\
        &= f_{k}(rx+y, r^2\widetilde{v}_{k}(x)+ Q(rx+y), rD\widetilde{v}_{k}(x)+DQ(rx+y))\\
        &:=\widetilde{f}_{k}(x,\widetilde{v}_{k},D\widetilde{v}_{k}), 
    \end{align*}
    where $\widetilde{f}_{k}(x,z,p)=f_{k}(rx+y, r^2z+ Q(rx+y) , rp + DQ(rx+y)).$
    
    Define the operator
    \[ G(M,p,z,x):=\sigma_{2}(M+D^2Q)-\sigma_{2}(D^2Q)-\widetilde{f}_{k}(x,z,p)+\widetilde{f}_{k}(x,0,0) \]
    for $(M,p,z,x)\in \mathcal{S}^{n\times n}\times \bb{R}^n \times \bb{R} \times B_1$.
    Clearly, $G$ satisfies the hypotheses H1)-H4) of Theorem \ref{THM: Generalized Savin's thm}. Let $\delta, C$ be the constants in Theorem \ref{THM: Generalized Savin's thm}. Now $\widetilde{v}_{k}$ solves 
    \begin{align*}
        G(D^2\widetilde{v}_{k},D\widetilde{v}_{k},\widetilde{v}_{k},x)&=\widetilde{f}_{k}(x,0,0)-\sigma_{2}(D^2Q)\\
        &= f_{k}(rx+y,Q(rx+y), DQ(rx+y))-f(y,Q(y),DQ(y))\\
        &:=\widehat{f}_{k}(x)
    \end{align*}
    Next, we verify the conditions in Theorem \ref{THM: Generalized Savin's thm}. 
    \begin{align*}
        \|\widetilde{v}_{k}\|_{L^{\infty}(B_{1})}&\leq \dfrac{\|u_{k}-u\|_{L^{\infty}(B_{r}(y))}}{r^2}+\dfrac{\|u-Q\|_{L^{\infty}(B_{r}(y))}}{r^2}\\
        &\leq \dfrac{\|u_{k}-u\|_{L^{\infty}(B_{r}(y))}}{r^2}+\sigma(r),
    \end{align*}
    for some modulus $\sigma(r)=o(r^2)/r^2$, 
    \begin{align*}
        |G(M,p,z,x)-G(M,p,z,x')|&\leq |\widetilde{f}_{k}(x,z,p)-\widetilde{f}_{k}(x',z,p)|+|\widetilde{f}_{k}(x,0,0)-\widetilde{f}(x',0,0)|\\
        &\leq C_{1}r^{\alpha}|x-x'|^{\alpha},
    \end{align*}
    and
    \begin{align*}
    \|\widehat{f}_{k}\|_{C^{0,\alpha}(B_1)}&\leq \|f_{k}(rx+y,Q(rx+y),DQ(rx+y))-f_{k}(y,Q(y),DQ(y))\|_{C^{0,\alpha}(B_1)}\\
    &\quad +|f_{k}(y,Q(y),DQ(y))-f(y,Q(y),DQ(y))|\\
    &\leq C_{2}r^{\alpha}+|f_{k}(y,Q(y),DQ(y))-f(y,Q(y),DQ(y))|,
    \end{align*}
    where $C_{1}$ depends on $A$ and $C_{2}$ depends on $A$, $|Du(y)|$, $|D^2u(y)|$. We emphasize that $C_{1}$ and $C_{2}$ are independent on $k$ and $r$. We fix $r=r(\delta ,\sigma,C_{1}, C_{2}):=\rho$ sufficiently small, such that $\sigma(\rho)+C_{1}\rho^{\alpha}+C_{2}\rho^{\alpha} \leq \delta/2$. Then for sufficiently large $k$, we have 
    \[ |G(M,p,z,x)-G(M,p,z,x')|\leq \delta|x-x'|^{\alpha}, \]
    and
    \[  \|\widetilde{v}_{k}\|_{L^{\infty}(B_{1})}\leq \delta,\qquad \|\widehat{f}_{k}\|_{C^{0,\alpha}(B_{1})}\leq\delta. \]
    Hence, by Theorem \ref{THM: Generalized Savin's thm}, we obtain $\|\widetilde{v}_{k}\|_{C^{2,\alpha}(B_{1/2})}\leq C$, where $C$ is independent on $k$.
    This implies  
    \[\Delta u_{k}\leq C\quad \text{on}\  B_{\rho/2}(y)\ \text{ uniformly in}\ k.\]
    
    \emph{Step 3.} Finally, we apply the doubling inequality finitely many times to propagate the uniform bound for $\Delta u_{k}$ to the origin. Indeed, denote $N=[\log_{2}\rho^{-1}]+1$.  By Prop \ref{PROP: Doubling inequality}, we obtain
    \[ \sup_{B_{1}(y)}\Delta u_{k}\leq \sup_{B_{2^{N}\rho}(y)}\Delta u_{k} \leq C_{0} \sup_{B_{2^{N-1}\rho}(y)}\Delta u_{k} \leq \cdots\leq  C_{0}C_{1}\cdots C_{N}\sup_{B_{\rho/2}(y)}\Delta u_{k}\leq \widetilde{C}, \]
    where $\widetilde{C}$ is independent on $k$. This contradicts with $iii)\ |D^2u_{k}(0)|\to\infty$. \hfill\qed

\bibliographystyle{amsalpha}
\bibliography{ref}
\end{document}